\documentclass[11pt]{amsart}

\issueinfo{00}{0}{Month}{Year}
\copyrightinfo{2008}{University of Calgary}
\pagespan{1}{2}
\PII{ISSN 1715-0868}

\newtheorem{theorem}{Theorem}
\newtheorem{lemma}[theorem]{Lemma}

\theoremstyle{remark}

\newtheorem{cor}[theorem]{\bf Corollary}

\newtheorem{prop}[theorem]{\bf Proposition}
\numberwithin{equation}{section}
\usepackage{enumitem}
\setlist[itemize,1]{label=$\bullet$}
\usepackage{esvect}

\usepackage{siunitx}
\usepackage{xstring} 
\begin{document}
\include{cdmstdcmds}
\title{The directed uniform Hamilton-Waterloo Problem involving even cycle sizes}

\author{FAT\.{I}H YETG\.{I}N}
\address{Department of Mathematics, Gebze Technical University, Kocaeli, 41400, Turkey}
\email{fyetgin@gtu.edu.tr}

\author{U\u{g}ur Odaba\c{s}{\i}}
\address{Department of Engineering Sciences, Istanbul University-Cerrahpasa, Istanbul, 34320, Turkey}
\email{ugur.odabasi@iuc.edu.tr}

\author{S\.{I}Bel \"OZKAN}
\address{Department of Mathematics, Gebze Technical University, Kocaeli, 41400, Turkey}
\email{s.ozkan@gtu.edu.tr}

\date{(date1), and in revised form (date2).}


\begin{abstract}
In this paper, factorizations of the complete symmetric digraph $K_v^*$ into uniform factors consisting of directed even cycle factors are studied as a generalization of the undirected Hamilton-Waterloo Problem. It is shown, with a few possible exceptions, that $K_v^*$ can be factorized into two nonisomorphic factors, where these factors are uniform factors of $K_v^*$ involving $K_2^*$ or directed $m$-cycles, and directed $m$-cycles or $2m$-cycles for even $m$.
\end{abstract}

\maketitle

\section{Introduction} \label{S:intro}
In this paper, edges and arcs are denoted by using curly braces and parentheses, respectively. Throughout this paper, we denote by $K_{(x:y)}$ a complete equipartite graph having $y$ parts of size $x$ each. Also, for a simple graph $G$, we use $G^*$ to denote the symmetric digraph with vertex set $V(G^*)=V(G)$ and arc set $E(G^*)=\bigcup_{\{x,y\}\in E(G)} \{(x,y),(y,x)\}$. Hence, $K_v^*$
and $K_{(x:y)}^*$ respectively denote the complete symmetric digraph of order $v$ and the complete symmetric equipartite digraph with $y$ parts of size $x$. We also use $(x, y)^*$ to denote the double arc which consists of $(x, y)$ and $(y, x)$.

A $k$-factor of a graph $G$ is a $k$-regular spanning subgraph of $G$. A $k$-factorization of a graph $G$ is a partition of the edge set of $G$ into $k$-factors, in other words it is a decomposition of edges set of $G$ into edge-disjoint $k$-factors. It is easy see that a $2$-factor consists of a cycle or union of vertex-disjoint cycles. There are two well-studied $2$-factorization problems. The Oberwolfach Problem asks for the existence of decomposition of the complete graph $K_v$ into the given $2$-factor $F$. The uniform version of the Oberwolfach Problem in which there is only one type of cycle in the factor $F$ has been mostly solved, see \cite{Alspach1985,Alspach1989,Hoffman1991}. In the Hamilton-Waterloo Problem, there are two types of $2$-factors. The uniform version of the Hamilton-Waterloo Problem asks for a $2$-factorization of $K_v$ (or for even $v$, $2$-factorization of $K_v-I$) in which $r$ of its $2$-factors consist of only $m$-cycles and remaining $s$ of its $2$-factors consist of only $n$-cycles, and we will denote it by HWP$(v; m^r, n^s)$. We also denote a solution as a $\{C_m^{r}, C_n^{s}\}$-factorization of $K_v$ (or $K_v-I$ for even $v$).

Initially, small cases such as $(m, n)\in \{(4,6),(4,8)$,$(4,16),(8,16),(3,5),\allowbreak(3,15),(5,15)\}$ are studied and solved with a few exceptions by Adams et al. \cite{Adams2002}, and later the cases where the cycle sizes are non-constant are investigated. The Hamilton-Waterloo Problem is nearly completely solved when  both $m$ and $n$ are simultaneously either even or odd \cite{burgess20182, burgess2017, Bryant2011, Bryant2013}. When the parity of $m$ and $n$ is different, one of the cycle sizes is usually fixed. For instance, the cases $(m,n)\in \{(3,v),(3,3x),(4,n)\}$ have been studied, see \cite{Asplund2016,  Fu2012, Keranen, Odabas2016}. For more recent results on this problem, we refer the reader to \cite{Burgess2022, burgess2018, burgess2019}. Also, there exists an asymptotic solution \cite {Glock} for the general form of the Oberwolfach and the Hamilton-Waterloo Problems.

The concept of factor and factorization can be applied to digraphs and one can consider as the directed version of the Oberwolfach and Hamilton-Waterloo Problems. In the directed version of these problems, factorization of the complete symmetric digraph $K_v^*$ into directed cycle factors is studied. The Directed Uniform Oberwolfach Problem is denoted by OP$^{*}(m^k)$ where each $2$-factor is composed of $k$ directed $m$-cycles.

The following theorem summarizes the previous results on the Directed Oberwolfach Problem that will be used in this paper.
\begin{theorem}\cite{Abel2002,Adamsun, Bermond1979, Bennett1990, Sajna2014,Sajna2018,Lacaze2023}\label{OP}
Let $m$ and $k$ be nonnegative integers. Then, $\mathrm{OP}^{*}(m^k)$ has a solution if and only if $(m, k) \notin \{(3,2), (4,1), (6,1)\}$.
\end{theorem}


When it comes to the Directed Hamilton-Waterloo Problem, here $K_v^*$ is decomposed into two types of directed $2$-factors. If these factors consist of directed cycles of sizes $m$ and $n$ respectively, the notation HWP$^{*}(v; m^{r}, n^{s})$ is used to denote the Directed Uniform Hamilton-Waterloo Problem.

In \cite{biz2022}, the necessary conditions for the existence of a solution to the Directed Hamilton-Waterloo Problem are given. 

\begin{lemma}\cite{biz2022}\label{necessary}
If $\mathrm{HWP}^{*}(v; m^{r}, n^{s})$ has a solution then the following statements hold;
\begin{enumerate}
\item if $r>0$, $v \equiv 0 \pmod m$, 
\item if $s>0$, $v \equiv 0 \pmod n$, 
\item $r+s=v-1$.
\end{enumerate} 
\end{lemma}
Additionally, the cases $(m, n)\in \{(3,5),(3,15),(5,15),(4,6),(4,8),(4,12),$ $(4,16),(6,12),(8,16)\}$ are solved with a few  possibly exceptions in \cite{biz2022}.

In \cite{odabas2020}, factorizations of $K_v$ into $K_2$-factors and $C_m$-factors is studied, and also new solutions to $\mathrm{HWP}(2m;$ $m^r,(2m)^s)$ are given. Here, the problem of decomposing $K_v^*$ into $K_2^*$-factors and $\vv{C}_m$-factors will be examined where $\vv{C}_m$ is the directed cycle of order $m$. Since $K_2^*$ can be considered as $\vv{C}_2$, this problem can be included in the $\mathrm{HWP}^{*}(v; 2^{r}, m^{s})$. Afterwards, $\mathrm{HWP}^{*}(v; m^{r}, (2m)^{s})$ will be studied.

In Section 2, we give some basic definitions and present some preliminary results that will be used in the next sections. In Section 3, we focus on finding solutions to $\mathrm{HWP}^{*}(v; 2^{r}, m^{s})$ for even $m$ with $r+s=v-1$. Also a solution is denoted as a  $\{(K_{2}^{*})^{r},\vv{C}_{m}^{s}\}$-factorization of $K_v^*$. In Section 4, we will concentrate on solving $\mathrm{HWP}^{*}(v; m^{r}, (2m)^{s})$ for even $m$ with $r+s=v-1$. Here are our main results.
\begin{theorem}\label{main}
Let $r$, $s$ be nonnegative integers, and let $m\geq 4$ be even. Then, $\mathrm{HWP}^{*}(v; 2^{r}, m^{s})$ has a solution if $m| v$, $r+s=v-1$, $s \neq 1$, $(r, v)\neq(0, 6)$, $(m, r, v)\neq (4,0,4)$, and one of the following conditions holds;
\begin{enumerate}
\item $m > 4$, $s\neq 3$ and $m\equiv 0 \pmod 4$,
\item $m > 4$, $\frac{v}{m}$ is even, $s\neq 3$ and $m\equiv 2 \pmod 4$,
\item $m=4$ and $v\equiv 0,8,16 \pmod {24}$,
\item $m=4$, $v\equiv 12 \pmod {24}$ and $s\notin \{3, 5\}$,
\item $m=4$, $v\equiv 4, 20 \pmod {24}$ and $r$ is odd.
\end{enumerate}
\end{theorem}

\begin{theorem}
Let $r$, $s$ be nonnegative integers, and let $m\geq 4$ be even. Then, $\mathrm{HWP}^{*}(v; m^{r}, (2m)^{s})$ has a solution if and only if $m| v$, $r+s=v-1$ and $v\geq 4$ except for $(s,v,m)\in\{(0,4,4),(0,6,3),(5,6,6)\}$, and except possibly when $s\in\{1,3\}$.
\end{theorem}
\section{Preliminary Results}

Let $G$ be a graph and $G_{0}, G_{1}, \ldots, G_{k-1}$ be $k$ vertex disjoint copies of $G$ with $v_{i} \in V\left(G_{i}\right)$ for each $v \in V(G)$. Let $G[k]$ denote the graph with vertex set $V(G[k])=V\left(G_{0}\right) \cup V\left(G_{1}\right) \cup \ldots \cup V\left(G_{k-1}\right)$ and edge set $E(G[k])=\left\{\{u_{i}, v_{j}\}: \{u, v\} \in\right.$ $E(G)$ and $0 \leq i, j \leq k-1\}$. It is easy see that there is an $H[k]$-factorization of $G[k]$ if the graph $G$ has an $H$-factorization.

Häggkvist used $G[2]$ to build $2$-factorizations that include even cycles \cite{Haggkvist1985}.

\begin{lemma}[Häggkvist Lemma]
Let $G$ be a path or a cycle with $n$ edges and let $H$ be $a$ 2-regular graph on $2n$ vertices with all components even cycle. Then $G[2]\cong G^{\prime} \oplus \mathrm{G}^{\prime \prime}$ where $G^{\prime} \cong G^{\prime \prime} \cong H$. Therefore, $G[2]$ has an $H$-decomposition.
\end{lemma}

If $G_{1}$ and $G_{2}$ are two edge-disjoint graphs with $V(G_1)=V(G_2)$, then we use $G_{1} \oplus G_{2}$ to denote the graph on the same vertex set 
with $E\left(G_{1} \oplus G_{2}\right)=E\left(G_{1}\right) \cup E\left(G_{2}\right)$. We will denote the vertex disjoint union of $\alpha$ copies of $G$ by $\alpha G$.

The above definitions can be extended to cover digraphs. Let D be a digraph and  $D_{0}, D_{1}, \ldots, D_{k-1}$ be $k$ vertex disjoint copies of $D$ with $v_{i} \in V\left(D_{i}\right)$ for each $v \in V(D)$. Then, $D[k]$ has the vertex set $V(D[k])=V\left(D_{0}\right) \cup V\left(D_{1}\right) \cup \cdots \cup V\left(D_{k-1}\right)$ and arc set $E(D[k])=\left\{(u_{i}, v_{j}): (u, v) \in\right.$ $E(D)$ and $0 \leq i, j \leq k-1\}$.

The following proposition, which is useful for transferring the results of undirected graphs to digraph and symmetric digraph, states that if we have an $H$-factorization of the undirected graph $G$, then using this factorization an $H^{*}$-factorization of $G^{*}$ can be obtained. 
\begin{prop}\label{lemma1.3}
Let $G$ be a graph and $H$ be a subgraph of $G$. If $G$ has an $H$-factorization, then $G^{*}$ has an $H^{*}$-factorization.
\end{prop}
It is known that $K_{2x}$ has a $1$-factorization \cite{Reiss1859}. Therefore, as a natural consequence of Proposition \ref{lemma1.3}, the following proposition can be stated.
\begin{prop} \label{lemma3}
The complete symmetric digraph $K_{2x}^{*}$ has a $K_{2}^{*}$-factoriza-tion for every integer $x\geq 1$. 
\end{prop}

The following result of Liu on equpartite graph has been helpful in solving the Oberwolfach and Hamilton-Waterloo problems. We will use this result to obtain a $\vv{C}_{m}$-factorization of $K_{(x: y)}^*$.
\begin{theorem}\cite{Liu2003} \label{liu}
The complete equipartite graph $K_{(x: y)}$ has a $C_{m}$-factoriza-tion for $m \geq 3$ and $x \geq 2$ if and only if $m\vert xy$, $x(y-1)$ is even, $m$ is even if $y=2$ and $(x, y, m) \neq(2,3,3),(6,3,3),(2,6,3),(6,2,6)$.
\end{theorem}

The necessary and sufficient condition for the existence of a $1$-factorization of a complete equipartite graph $K_{(x: y)}$ is given by Hoffman and Rodger \cite{Hoffman1992}.

\begin{theorem}\cite{Hoffman1992}\label{unknown}
The complete equipartite graph $K_{(x: y)}$ has a $1$-factoriza-tion if and only if $xy$ is even.
\end{theorem}
By Proposition \ref{lemma1.3} and Theorem \ref{unknown}, we can say that $K_{(x: y)}^*$ has a $K_{2}^{*}$-factorization  for even $xy$.
\begin{lemma}\label{dliu2}
The complete symmetric equipartite digraph $K_{(x: y)}^*$ has a $K_{2}^{*}$-factorization  if and only if even $xy$.
\end{lemma}

We will also use the following two well-known results of Walecki.
\begin{lemma}\cite{Walecki}\label{wlackiodd}
For all odd $m \geq 3$, $K_{m}$ decomposes into $\big(\frac{m-1}{2}\big)$ Hamilton cycles.
\end{lemma}

\begin{lemma}\cite{Walecki}\label{wlackieven}
For all even $m \geq 4, K_{m}-F_m$ has an Hamilton cycle decomposition with prescribed cycles $\{C, \sigma\left(C\right)$, $ \sigma^{2}\left(C\right) \ldots, \sigma^{\frac{m-4}{2}}\left(C\right)\}$ for some permutation $\sigma$ of $\{0,1, \ldots, m-1\}$ where $C=\left(0,1, \ldots, m-1\right)$  and $E\left(F_m\right)=$ $\{\{0, m / 2\},\{i, m-i\}: 1 \leq i \leq(m / 2)-1\}$.
\end{lemma}

Lemmata \ref{lemma1.7} and \ref{lemma2.8} show the existence of the $\left\{C_{m}^{r}, C_{2 m}^{s}\right\}$-factorization of the $C_{m}[2]$ and $(C\oplus F_{m})[2]$ for $r+s\in \{2,3\}$. They will be used to find a $\{(K_{2}^{*})^{r}, \vv{C}_{m}^{s}\}$-factorization of the $C_{m}^{*}[2]$ for $r\in \{0,2,4\}$, $r+s=4$ and a $\vv{C}_{2m}$-factorization of $C^{*}[2]\oplus F_m^{*}[2]$ where $C^{*}$ is the symmetric version of the $C$ defined in Lemma \ref{wlackieven}. Also, we use $\Gamma_m$ and $\Gamma_m^{*}$ to denote $C[2]\oplus F_m[2]$ and $C^{*}[2]\oplus F^{*}_m[2]$, respectively, for the rest of the paper.

\begin{lemma}\cite{odabas2020} \label{lemma1.7}
Let $m$ be an integer with $m \geq 3$. Then $C_{m}[2]$ has a $\left\{C_{m}^{r}, C_{2 m}^{s}\right\}$-factorization for nonnegative integers $r$ and $s$ with $r+s=2$ except when $m$ is odd and $r=2$, and except possibly when $m$ is even and $r=1$.
\end{lemma}

\begin{lemma}\cite{odabas2020} \label{lemma2.8}
Let $m \geq 4$ be an even integer and $\Gamma_m$ where $C=$ $(0,1, \ldots, \allowbreak m-1)$ is an $m$-cycle and $F_{m}$ is a $1$-factor of $K_{m}$ with $E\left(F_{m}\right)=$ $\{\{0, m / 2\},\{i, \allowbreak m-i\}: 1 \leq i \leq(m / 2)-1\}$. Then $\Gamma_m[2]$ has $a$
\begin{enumerate}
\item $C_{2 m}$-factorization,
\item $C_{m}$-factorization when $m \equiv 0 \pmod 4$, and
\item $\left\{C_{m}^{2}, C_{2 m}^{1}\right\}$-factorization when $m \equiv 2 \pmod 4$.
\end{enumerate}
\end{lemma}

\begin{lemma}\cite{Sajna2014} \label{lemma2}
 Let $m \geq 4$ be an even integer and $x$ be a positive integer. Then, $K_{(\frac{mx}{2}:2)}^{*}$ has a $\vv{C}_{m}$-factorization.
\end{lemma}

\begin{theorem}\cite{Bennett1997} \label{c3thm}
The complete symmetric equipartite digraph $K_{(x:y)}^*$ has a $\vv{C}_{3}$-factorization if and only if $3 | x y$ and $(x, y) \neq(1,6)$ with possible exceptions $(x, y)=(x, 6)$, where $x\notin \{m:m \text{ is divisible by a prime less than} \,\ 17\}$.
\end{theorem}

The following theorem presents a solution for the Directed Hamilton-Waterloo problem for small even cycle factors. It will also help us in solving $\mathrm{HWP}^{*}(v; m^{r},$ $2m^{s})$ in Section $4$, when $m=4$.
\begin{theorem}\cite{biz2022}\label{ilk}
For nonnegative integers $r$ and $s$, $\mathrm{HWP}^{*}(v; m^{r}, n^{s})$ has a solution for $(m,n)\in \{(4,6),(4,8),(4,12),(4,16),(6,12),(8,16)\}$ if and only if $r+s=v-1$ and $\textnormal{lcm}(m,n)| v$.
\end{theorem}

Let $A$ be a finite additive group and let $S$ be a subset of $A$, where $S$ does not contain the identity of $A$. The Directed Cayley graph $\vv{X}(A ; S)$ on $A$ with connection set $S$ is a digraph with $V(\vv{X}(A ; S))=A$ and $E(\vv{X}(A ; S))=\{(x,y):x,y\in A, y-x\in S\}$.

Let $G$ be a digraph and $R(G)$ denote the digraph on the same vertex set as $G$ but the arcs are taken in opposite directions.

Let $m$ be an even integer and the vertex set of $K_{2m}^*$ be $\mathbb{Z}_{2m}$. Let $I_{2m}^*$ be a $K_2^*$-factor of $K_{2m}^*$ with $E\left(I_{2m}^*\right)=$ $\{(i, m +i)^*: 0 \leq i \leq m-1\}$ and define the bijective function $f:\mathbb{Z}_{2m} \rightarrow \mathbb{Z}_{2} \times \mathbb{Z}_{m}$ with $$
 f(i)= \begin{cases}
 (0, i) & \text { if } i < m, \\ (1, i-m) & \text { if }  i \geq m.\end{cases}
$$
Then, $E\left(I_{2m}^{*}\right)$ can be restated as a set $\Big\{\big((0, i), (1, i)\big)^*: 0 \leq i \leq m-1\Big\}$ on $\mathbb{Z}_{2} \times \mathbb{Z}_{m}$ using this bijective function. We will represent $C_{m}^{*}[2]$ and $C_{m}^{*}[2] \oplus I_{2m}^*$ as the directed Cayley graphs $\vv{X}\big(\mathbb{Z}_{2} \times \mathbb{Z}_{m}, S\big)$ and $\vv{X}\big(\mathbb{Z}_{2} \times \mathbb{Z}_{m},$ $S\cup \{(1,0)\}\big)$ where $S=\{(0,1), (1, 1), (0,-1), (1,-1)\}$.

Also, we define a factor $F_m^*$ as a $K_2^*$-factor of $K_{m}^*$ with $E\left(F_m^*\right)=$ $\{(0, m / 2)^*,$ $(i, m-i)^*: 1 \leq i \leq(m / 2)-1\}$. The arc set of $F_m^*$ which is denoted by $E\left(F^{*}_m\right)$, can be expressed as $\big\{\big((0,0), (0, m / 2)\big)^*,  \big((0, i), (0, m-i)\big)^*: 1 \leq i \leq(m/2)-1\big\}$ using above bijective function. So, $\Gamma_m$ can be represented as the directed Cayley graph $\vv{X}\big(\mathbb{Z}_{2} \times \mathbb{Z}_{m},$ $\bigcup_{i=1}^{(m/2)-1}\{(0, m-2i), (0, 2i-m), (1, m-2i), (1, 2i-m)\}\cup S \cup \{ (0, m/2), (0, -m/2), (1, m/2), (1, -m/2)\}\big)$.

Using Proposition \ref{lemma1.3}, Lemmata \ref{wlackiodd} and \ref{wlackieven}, we will obtain a $\{(C_{\frac{m}{2}}^{*}[2])^{\frac{m-2}{4}},$ $I_{m}^*\}$-factorization and a $\{(C_{\frac{m}{2}}^{*}[2])^{\frac{m-8}{4}},  I_{m}^*, \Gamma^{*}_\frac{m}{2}\}$-factorization of $K_{m}^{*}$ depending on whether $m\equiv 0 \,\ \text{or} \,\  2 \pmod 4$, then use these factorizations to obtain a $\{(K_{2}^{*})^{r},\vv{C}_{m}^{s}\}$-factorization of $K_{mx}^{*}$. Since, $I_{m}^*$ and $F_{\frac{m}{2}}^*$ do not contain any $m$-cycles, we will factorize $C_{\frac{m}{2}}^{*}[2]\oplus I_{m}^*$ and $\Gamma^{*}_\frac{m}{2}$ into $K_{2}^{*}$-factors and $\vv{C}_{m}$-factors. Furthermore, it is necessary to have a $\{(K_{2}^{*})^{r},\vv{C}_{2m}^{s}\}$-factorization of $C_{m}^{*}[2]$ in order to factorize $K_{mx}^{*}$ into $K_{2}^{*}$-factors and $\vv{C}_{m}$-factors.
\begin{lemma}\label{lemma4}
Let $m\geq 4$ be an integer. Then $C_{m}^{*}[2]$ has a $\{(K_{2}^{*})^{r},\vv{C}_{2m}^{s}\}$-factorization for $r\in \{0,2,4\}$ and $r+s=4$.
\end{lemma}
\begin{proof}
First, note that $C_m[2]$ has a decomposition into two $C_{2m}$-factors by Häggkvist Lemma and each $C_{2m}$-factor has a decomposition into two $1$-factors.

\textbf{Case 1} {\boldmath $(r=4)$} Decompose $C_m[2]$ into four $1$-factors by using $C_{2m}$-factors. Then a $K_{2}^{*}$-factorization of $C_{m}^{*}[2]$ is obtained by Proposition \ref{lemma1.3}.

\textbf{Case 2} {\boldmath $(r=2)$}  Decompose $C_m[2]$ into one $C_{2m}$ and two $1$-factors. By using Proposition \ref{lemma1.3}, we get a $\{(K_{2}^{*})^{2}, {C}_{2m}^*\}$-factorization of $C_{m}^{*}[2]$ and also $C_{2m}^{*}$ has a $\vv{C}_{2m}$-factorization with two $\vv{C}_{2m}$-factors. So, we obtain a $\{(K_{2}^{*})^{2},\vv{C}_{2m}^{2}\}$-factorization of $C_{m}^{*}[2]$.

\textbf{Case 3} {\boldmath $(r=0)$} Obtain a ${C}_{2m}^*$-factorization of $C_m^{*}[2]$ by Proposition \ref{lemma1.3}. Since $C_{2m}^{*}$ has a $\vv{C}_{2m}$-factorization with two $\vv{C}_{2m}$-factors, $C_{m}^{*}[2]$ has a $\vv{C}_{2m}$-factorization.
\end{proof}

Since $I_{2m}^*$ and $F_m^*$ are $K_2^*$-factors, the following result can be derived from Lemma \ref{lemma4}.
\begin{cor}\label{lemmagamma}
Let $m\geq 4$ be an even integer. Then $\Gamma^{*}_m$ has a $\{(K_{2}^{*})^{r}, \vv{C}_{2m}^{s}\}$-factorization for $r\in\{0,2,4,6\}$ with $r+s=6$.
\end{cor}
\begin{proof}
$F_m^*[2]$ decomposes into two $K_{2}^{*}$-factors. So, $\Gamma^{*}_m$ has a $\{(K_{2}^{*})^{r}, \vv{C}_{2m}^{s}\}$-factorization for $r\in\{2,4,6\}$ with $r+s=6$ by Lemma \ref{lemma4}. Also, $\Gamma^{*}_m$ has a $\vv{C}_{2m}$-factorization by Lemma \ref{lemma2.8} and Proposition \ref{lemma1.3}.
\end{proof}

The following lemma is quite useful in solving the Directed Hamilton-Waterloo Problem for  $n=2$ and even $m$ when the values of $r$ are even.

\begin{lemma}\label{lemma2.6}
Let $m\geq 5$ be an integer. Then $C_{m}^{*}[2] \oplus I_{2m}^*$ has a $\{(K_{2}^{*})^{r},\vv{C}_{2m}^{s}\}$-factorization for $r\in \{0,1,3,5\}$ and $r+s=5$.
\end{lemma}
\begin{proof} The cases $r\in\{1,3,5\}$ can be directly obtained from Lemma \ref{lemma4}.

When $r=0$, we will examine the problem in two cases; $m$ is odd or even.

\textbf{Case 1 (odd} {\boldmath $m\geq 5$\textbf{)}} 

Define five directed $2m$-cycles in $C_{m}^{*}[2] \oplus I_{2m}^*$ as follows:
$
\vv{C}_{2 m}^{(0)}=(v_{0}, v_{1}, \ldots,$ $ v_{2 m-1}) 
$
 where  $v_{i}=(\lfloor\frac{i}{m} \rfloor, i)$,
$
\vv{C}_{2 m}^{(1)}=\left(u_{0}, u_{1}, \ldots, u_{2 m-1}\right) \text { where } 
$
$$
 u_{2 i}= \begin{cases}(0,2 i) & \text { if } 0 \leq i \leq \frac{m-1}{2}, \\ (0,-2 i-1) & \text { if } \quad \frac{m+1}{2} \leq i \leq m-1,\end{cases} 
 $$
and
 $$
 u_{2i+1}= \begin{cases}(1,2 i+1) & \text { if } 0 \leq i \leq \frac{m-3}{2}, \\ (1,-2i-2) & \text { if } \quad \frac{m-1}{2} \leq i \leq m-1.\end{cases} 
$$ $C_{2 m}^{(2)}=\left(x_{0,}, x_{1}, \ldots, x_{2 m-1}\right) $ where
$$
x_{i}= \begin{cases}(0, m-\lfloor\frac{i}{2}\rfloor) & \text { if } i \equiv 0,3 \pmod 4 \\ (1, m-\lfloor\frac{i}{2}\rfloor) & \text { if } i\equiv1,2 \pmod 4\end{cases} \,\, \text {for} \,\, 0 \leq i \leq 2 m-3,
$$
and $x_{2 m-2}=(1,1)$, $x_{2 m-1}=(0,1)$. Also, 
$
C_{2 m}^{(3)}=\left(y_{0,}, y_{1}, \ldots, y_{2 m-1}\right)  
$
where
$$
y_i=u_i+(1,2)\,\ \text {for} \,\  0 \leq i \leq m-3 \,\ and \,\ m+2 \leq i \leq 2m-1,
$$
$$
y_{m-2}=(1,0), \,\ y_{m-1}=(0,1), \,\  y_{m}=(1,1), \,\ y_{m+1}=(0,0).
$$

Finally, $\vv{C}_{2 m}^{(4)}=(C_{m}^{*}[2] \oplus I_{2m}^*)-\bigoplus_{i=0}^3 \vv{C}_{2 m}^{(i)}$. Then, $\{\vv{C}_{2 m}^{(0)}, \vv{C}_{2 m}^{(1)},$ $\vv{C}_{2 m}^{(2)}, \vv{C}_{2 m}^{(3)}, \allowbreak \vv{C}_{2 m}^{(4)}\}$ is a $\vv{C}_{2 m}$-factorization of $C_{m}^{*}[2] \oplus I_{2m}^*$.

\textbf{Case 2 (even} {\boldmath $m\geq 6$\textbf{)}} 

Let $\vv{C}_{2 m}^{(0)}$ be the same as in Case 1 and define the directed $2m$-cycles in $C_{m}^{*}[2] \oplus I_{2m}^*$ as follows:\\
$\vv{C}_{2m}^{(1)}=\left(x_{0}, x_{1}, \ldots, x_{2m-1}\right)$ where $x_{0}=(0,0)$ and 
$$
x_{i}=\begin{cases}
\left(0, m-\left\lfloor \frac{i+2}{2}\right\rfloor\right) & \text {\,\ if } i \equiv 1,2 \pmod 4 \\
(1, m-\lfloor\frac{i+2}{2}\rfloor+1) & \text { if } i \equiv 0,3 \pmod 4 
\end{cases} \,\, \text { for } \,\, 1 \leq i \leq 2 m-8,
$$
and $x_{2 m-6+2 i}=(0,3-i) \text { for } 0 \leq i \leq 2$ and $x_{2 m-7+2 i}=(1,3-i) \text { for } 0 \leq i \leq 3$. Also, $\vec{C}_{2 m}^{(2)}=\left(u_{0}, u_{1}, \ldots, u_{2 m-1}\right)$ where $u_{0}=(0,0), \,\ u_{1}=(1,0), \,\ u_{2}=(0, m-1)$ and 
$$
u_{i}=\left\{\begin{array}{lll}
\left(0, m-\lfloor \frac{i-1}{2}\rfloor-1\right) & \text { if }  i \equiv 0,1 \pmod4 \\
\left(1, m-\lfloor  \frac{i-1}{2}\rfloor\right) & \text { if }  i \equiv 2,3 \pmod4
\end{array} \quad \right.  \,\, \text { for } \,\, 3 \leq i \leq 2m-9,
$$
$
u_{2 m-8+j}=\left\{\begin{array}{l}
(0,4-\lfloor\frac{j}{2}\rfloor) \text { if } j \equiv 0,2\pmod4 \\
(1,4-\lfloor\frac{j}{2}\rfloor) \text { if } j\equiv 1,3 \pmod4 
\end{array}\right.  \,\, \text { for } \,\, 0 \leq j \leq 7,
$
and when $m=6$, $u_3=(1,5)$ and we only use above piecewise function. $\vv{C}_{2m}^{(3)}=\left(y_{0}, y_{1}, \ldots, y_{2m-1}\right)$ where $y_{2 i+2}=(0, m-i)$ for $ 1 \leq i \leq m-4$, $y_{2 i+1}=(1, m-i)$ for $ 1 \leq i \leq m-3$, $y_{0}=(0,0)$, $y_{1}=(1,1)$, $y_{2}=(1,0)$, $y_{2 m-4}=(1,2)$, $y_{2 m-3}=(0,3)$, $y_{2 m-2}=(0,2)$ and $y_{2 m-1}=(0,1)$. $\vec{C}_{2 m}^{(4)}=\left(z_0,z_1 \ldots, z_{2 m-1}\right)$ where $z_{9+2 i}=(0,4+i)$ for $1 \leq i \leq m-5$, $z_{10+2 i}=(1,4+i)$ for $0 \leq i \leq m-6$, $z_{0}=(0,0)$, $z_{1}=(1, m-1)$, $z_{2}=(1,0)$, $z_{3}=(0,1)$, $z_{4}=(1,2)$, $z_5=(1,1)$, $z_{6}=(0,2)$, $z_{7}=(1,3)$, $z_{8}=(0,4)$, $z_{9}=(0,3)$. Then $\left\{\vv{C}_{2 m}^{(0)}, \vv{C}_{2 m}^{(1)},\vv{C}_{2 m}^{(2)}, \vv{C}_{2 m}^{(3)}, \vv{C}_{2 m}^{(4)}\right\}$ is a $\vv{C}_{2 m}$-factorization of $C_{m}^{*}[2] \oplus I_{2m}^*$.
\end{proof}

By Lemma \ref{lemma1.7}, we can decompose $C_{m}[2]$ into two ${C}_{m}$-factors for even $m$. So, we obtain the following lemma similar to Lemma \ref{lemma4}. Also, the following Corollaries are obtained as a result of this lemma.
\begin{lemma}\label{lemma2.4}
Let $m\geq 4$ be an even integer. Then $C_{m}^{*}[2]$ has a  $\{(K_{2}^{*})^{r}, \vv{C}_{m}^{s}\}$-factorization for $r\in \{0,2,4\}$ with $r+s=4$.
\end{lemma} 

\begin{cor}\label{lemma2.5}
Let $m\geq 4$ be an even integer. Then $C_{m}^{*}[2]\oplus I_{2m}^*$ has a  $\{(K_{2}^{*})^{r}, \vv{C}_{m}^{s}\}$-factorization for $r\in \{1,3,5\}$ with $r+s=5$.
\end{cor} 

\begin{cor}\label{lemma3.6}
Let $m\geq 4$ be an even integer. Then $\Gamma^{*}_m$ has a  $\{(K_{2}^{*})^{r}, \vv{C}_{m}^{s}\}$-factorization for $r\in \{2,4,6\}$ with $r+s=6$.
\end{cor} 

Theorem \ref{liu} states that $K_{(x: y)}$ has a $C_{m}$-factorization with a few exceptions. We will use this result to show that $K_{(x: y)}^*$ has a $\vv{C}_{m}$-factorization. However, some of the exceptions in the undirected version do not exist in the symmetric version. It is shown that there is actually a solution for these exceptions in the symmetric version. As a corollary of Lemma \ref{lemma2}, Proposition \ref{lemma1.3} and Theorems \ref{liu} and \ref{c3thm}, we can give the following result.

\begin{lemma}\label{dliu}
The complete symmetric equipartite digraph $K_{(x: y)}^*$ has a $\vv{C}_{m}$-factorization for $m \geq 3$ and $x \geq 2$ if $m\vert xy$, $x(y-1)$ is even, $m$ is even when $y=2$.
\end{lemma}
\begin{proof}
Let $m\vert xy$, $x(y-1)$ be even, $m$ be even when $y=2$, and let $(x, y, m) \notin \{(2,3,3),(6,3,3),(2,6,3), (6,2,6)\}$. By Theorem \ref{liu}, $K_{(x: y)}$ has a $C_{m}$-factori-zation and so, $K_{(x: y)}^{*}$ has a $C_{m}^*$-factorization by Proposition \ref{lemma1.3}. Since each $C_{m}^*$ has a $\vv{C}_{m}$-factorization, $K_{(x: y)}^{*}$ has a $\vv{C}_{m}$-factorization. Since $K_{(x: y)}^*$ has a $\vv{C}_{m}$-factorization for $(x, y, m) \in \{(2,3,3),(6,3,3),(2,6,3), (6,2,6)\}$ by Theorem \ref{c3thm} and Lemma \ref{lemma2}, we conclude that $K_{(x: y)}^{*}$ has a $\vv{C}_{m}$-factorization for $m \geq 3$ and $x \geq 2$ if $m\vert xy$, $x(y-1)$ is even, $m$ is even when $y=2$.
\end{proof}

Recall that $\Gamma^{*}_m$ is $C^{*}[2]\oplus F_m^{*}[2]$.
\begin{lemma}\label{lemma3.10}
$\Gamma^{*}_m$ has a $\{(K_{2}^{*})^{r}, \vv{C}_{m}^{s}\}$-factorization for $m\equiv 2 \pmod 4$ and $r\in\{1,2,3,4,6\}$ with $r+s=6$.
\end{lemma}
\begin{proof}
The cases $r\in\{2,4,6\}$ are obtained by Corollary \ref{lemma3.6}.

For $r=1$, we define the following $m$-cycles.
$\vv{C}_{m}^{(0)}=\left(v_{0}, v_{1}, \ldots v_{m-1}\right)$ where $v_{i} =(0, i)$ for $0 \leq i \leq m-1$.\\
$
\vv{C}_{m}^{(1)}=\left(u_{0}, u_{1}, \ldots, u_{m-1}\right) \text { where } u_{i}= \begin{cases}(0, i) & \text { if }  \text {i is even,} \\ (1, i) & \text { if } i \text { is odd.}\end{cases}
$
\\
$
\vv{C}_{m}^{(2)}=(x_{0}, x_{1}, \ldots x_{m-1} )\text { where } x_{0}=(0,0) \text { and for } 1 \leq i\leq m-1   
$
$$
x_{i=}\left\{\begin{array}{l}
\left(\frac{1-(-1)^{i}}{2}, \frac{m}{2}-\lfloor \frac{i}{2}\rfloor\right) \text { if } i \equiv 1, 2 \pmod4, \\
\left(\frac{1-(-1)^{i}}{2}, \frac{m}{2}+\lfloor \frac{i}{2}\rfloor\right) \text { if } i=0, 3\pmod4.
\end{array}\right.
$$

Let's choose the factor $F_0$ as isomorphic to $F_m^*\oplus(F_m^*+(1,0))$, then $F_0$ becomes a $K_{2}^{*}$-factor. Using the above $m$-cycles, we obtain the following $m$-cycle factors. $F_{1}=\vec{C}_{m}^{(0)} \cup(\vv{C}_{m}^{(0)}+(1,0))$, $F_{2}=R\left(F_{1}\right)$, $F_{3}=\vv{C}_{m}^{(1)} \cup R(\vv{C}_{m}^{(1)} +(1,0))$, $F_{4}=\vv{C}_{m}^{(2)} \cup R(\vv{C}_{m}^{(2)}+(1, 0))$ and $F_{5}=\Gamma^{*}_m-\bigoplus_{i=0}^{4} F_{i}$. Then, $\left\{F_{0}, F_{1}, F_{2}, F_{3}, F_{4},  F_{5}\right\}$ is a $\{(K_{2}^{*})^{1}, \vv{C}_{m}^{5}\}$-factorization of $\Gamma^{*}_m$.

For $r=3$, $F_{1} \oplus F_{2}$ is a $C_{m}^{*}$-factor of $\Gamma^{*}_m$ and has a factorization into two $K_{2}^*$-factors of $\Gamma^{*}_m$ say $F_{1}^{'}$ and $F_{2}^{'}$. Then $\left\{F_{0}, F_{1}^{'}, F_{2}^{'}, F_{3}, F_{4}, F_{5}\right\}$ is a $\{(K_{2}^*)^{3}, \vv{C}_{m}^{3}\}$-factorization of $\Gamma^{*}_m$.
\end{proof}
\section{Solutions to HWP$^{*}(v;2^r, m^s)$}
Now, we can give solutions to the Directed Hamilton-Waterloo Problem for $K_{2}^*$ and $\vv{C}_{m}$ when even $m$.

\begin{theorem}\label{maintheorem}
Let $r$, $s$ be nonnegative integers, and let $m\geq 6$ be even. Then, $\mathrm{HWP}^{*}(v; 2^{r}, m^{s})$ has a solution if and only if $m| v$, $r+s=v-1$ and $v\geq 6$ except for $s = 1$ and $(r, v)=(0, 6)$, and except possibly when at least one of the following conditions holds;
\begin{enumerate}
\item $s=3$ and $m\equiv 0 \pmod 4$,
\item $s=3$, $m\equiv 2 \pmod 4$ and $\frac{v}{m}$ is odd.
\end{enumerate}
\end{theorem}

\begin{proof} 
Take $(v-2)$ disjoint $K_2^*$-factors of $K_v^*$, say $H_1^*,H_2^*,\dots, H_{v-2}^*$. It is obvious that $K_v^*-(H_1^*\oplus H_2^*\oplus\dots\oplus H_{v-2}^*)$ is a $K_2^*$-factor in $K_v^*$. Thus, there is no $\{(K_2^*)^{v-2},\vv{C}_m^1 \}$-factorization of $K_v^*$. So, we may assume $s \neq 1$.

Since $\mathrm{HWP}^{*}(v; n^{r}, m^{s})$ has a solution for $r=0$ except for $(v,m)=(6,6)$ by Theorem \ref{OP}, we may assume that $r\geq 1$.

Let $v=mx$ for a positive integer $x$. Partition the vertices of $K_{mx}^{*}$ into $2x$ sets of size $\frac{m}{2}$, represent each part of $\frac{m}{2}$ vertices in $K_{mx}^{*}$ with a single vertex and represent all double arcs between sets of size $\frac{m}{2}$ as a single double arc, to get a $K_{2x}^{*}$. By Proposition \ref{lemma3}, $K_{2x}^{*}$ has a decomposition into $(2x-1)$ $K_{2}^{*}$-factors. Then, construct a $K_{m}^{*}$-factor of $K^{*}_{mx}$ from one of the $K_{2}^{*}$-factors, and a $K_{(\frac{m}{2}:2)}^{*}$-factor of $K^{*}_{mx}$ from each of the remaining $(2x-2)$  $K_{2}^{*}$-factors. Then, $K^{*}_{mx}$ can be factorized into a $K_{m}^{*}$-factor and $(2x-2)$ $K_{(\frac{m}{2}:2)}^{*}$-factors.

$K_{(\frac{m}{2}:2)}^{*}$ decomposes into $\frac{m}{2}$ $K_{2}^{*}$-factors or $\frac{m}{2}$ $\vv{C}_{m}$-factors by Lemmata \ref{dliu2} and \ref{lemma2}, respectively. So, we must decompose $K_{m}^{*}$ into $K_{2}^{*}$-factors and $\vv{C}_{m}$-factors.

\textbf{Case 1 (odd \boldmath{$r$})} 

By Lemma \ref{wlackieven}, factorize $K_{m}$ into a $F_{m}$-factor and $(\frac{m-2}{2})$ $C_{m}$-factors. So, $K^{*}_{m}$ can be factorized into a $F_{m}^{*}$-factor and $(\frac{m-2}{2})$ $C_{m}^{*}$-factors by Proposition \ref{lemma1.3}.

Since $C_{m}^{*}$ can be decomposed into two $K_{2}^{*}$-factors or two $\vv{C}_{m}$-factors for even $m$, $K_{m}^{*}$ has a $\{(K_{2}^{*})^{2r_1+1},\vv{C}_{m}^{2s_1}\}$-factorization where $r_1+s_1=\frac{m-2}{2}$.

Since $K_{mx}^{*}$ has a $\{K_{m}^{*}, (K_{(\frac{m}{2}:2)}^{*})^{(2x-2)}\}$-factorization, placing a $K_{2}^{*}$-factoriz-ation on $r_0$ of the $K_{(\frac{m}{2}:2)}^{*}$ factors for $r_0$ even and $0\leq r_0\leq 2x-2$, a $\vv{C}_{m}$-factorization on $s_0$ of the $K_{(\frac{m}{2}:2)}^{*}$ where  $r_0+s_0=2x-2$, and taking a $\{(K_{2}^{*})^{2r_1+1},\vv{C}_{m}^{2s_1}\}$-factorization of $K_{m}^{*}$ give a $\{(K_{2}^{*})^{r},\vv{C}_{m}^{s}\}$-factorization of $K_{mx}^{*}$ where $r=\frac{m}{2}r_0+2r_1+1$ and $s=\frac{m}{2}s_0+2s_1$ with $r+s=\frac{m}{2}(r_0+s_0)+2(r_1+s_1)+1=mx-1=v-1$. 

Since any nonnegative odd integer $1\leq r\leq mx-1$ can be written as $r=\frac{m}{2}r_0+2r_1+1$ for integers $0\leq r_0\leq 2x-2$ and $0\leq r_1\leq \frac{m-2}{2}$, a solution to $\mathrm{HWP}^{*}(v; 2^{r}, m^{s})$ exists for each odd $r \geq 1$ and  $s \geq 1$ satisfying $r+s=mx-1=v-1$.

\textbf{Case 2 (even \boldmath{$r$})} 

\textbf{(a)} Assume $m\equiv 0 \pmod 4$. So, $\frac{m}{2}$ is even. Each $K_{(\frac{m}{2}:2)}^{*}$ decompose into $\frac{m}{2}$ $K_{2}^{*}$-factors or $\frac{m}{2}$ $\vv{C}_{m}$-factors. So, we need a $\{(K_{2}^{*})^{r}, \vv{C}_{m}^{s}\}$-factorization of $K_{m}^{*}$ for even $r$.

Also, $K_{\frac{m}{2}}^{*}$ can be factorized as $\bigoplus_{i=1}^{\frac{m-8}{4}} C^{*}_{i} \oplus \Gamma^{*}_{\frac{m}{2}}$ where each $C^{*}_{i}$ is isomorphic to $C_{\frac{m}{2}}^{*}$. Then, $K_{\frac{m}{2}}^{*}[2]  \cong  \bigoplus_{i=1}^{\frac{m-8}{4}} C^{*}_{i} [2] \oplus  \Gamma^{*}_{\frac{m}{2}}[2]  
$. Also, $K^{*}_{m}$ is isomorphic to $K_{\frac{m}{2}}^{*}[2] \oplus I_{m}^*$. Therefore, $K_{m}^{*}$ has a $\{(C_{\frac{m}{2}}^{*}[2])^{\frac{m-12}{4}},  C_{\frac{m}{2}}^{*}\oplus I_{m}^*,  \Gamma^{*}_{\frac{m}{2}}\}$-factorization. By Lemma \ref{lemma4}, each of $\frac{m-12}{4}$ $C_{\frac{m}{2}}^{*}[2]$-factors has a $\{(K_{2}^{*})^{r_0},\vv{C}_{m}^{s_0}\}$-factorization for $r_0\in \{0,2,4\}$ and $r_0+s_0=4$. By Lemma \ref{lemma2.6}, $C_{\frac{m}{2}}^{*}[2]\oplus I_{m}^*$ has a $\{(K_{2}^{*})^{r_1},\vv{C}_{m}^{s_1}\}$-factorization for $r_1\in \{0,1,3,5\}$ and $r_1+s_1=5$.
By Corollary \ref{lemmagamma}, $\Gamma^{*}_{\frac{m}{2}}$ has a  $\{(K_{2}^{*})^{r_2}, \vv{C}_{m}^{s_2}\}$-factorization for even $m$ and $r_2\in\{0,2,4,6\}$ with $r_2+s_2=6$. Those factorizations give a $\{(K_{2}^{*})^{r'}, \vv{C}_{m}^{s'}\}$-factorization of $K_{m}^{*}$ where $r'=(\frac{m-12}{4})r_0+r_1+r_2$ and $s'=(\frac{m-12}{4})s_0+s_1+s_2$ satisfying $r'+s'=(\frac{m-12}{4})4+5+6=m-1$ with $0\leq r', s' \leq m-1$. If we choose $r_1=0$ , we obtain a $\{(K_{2}^{*})^{r'}, \vv{C}_{m}^{s'}\}$-factorization of $K_{m}^{*}$ for even $r'$.

Placing a $K_{2}^{*}$-factorization on $r''$ of the $K_{(\frac{m}{2}:2)}^{*}$-factors for $0\leq r''\leq 2x-2$, a $\vv{C}_{m}$-factorization on $s''$ of the $K_{(\frac{m}{2}:2)}^{*}$ for $r''+s''=2x-2$, and taking a $\{(K_{2}^{*})^{r'}, \vv{C}_{m}^{s'}\}$-factorization of $K_{m}^{*}$ give a $\{(K_{2}^{*})^{\frac{m}{2}r''+r'},\vv{C}_{m}^{\frac{m}{2}r''+s'}\}$-factorization of $K_{mx}^{*}$ where $\frac{m}{2}r''+r'$ is even. It can be seen that $r'=m-4$ cannot be obtained for the possible values of $r_0,r_1$ and $r_2$ from the above factorizations.

Since any even integer $1\leq r\leq mx-1$ can be written as $r=\frac{m}{2}r''+r'$ except for $r=mx-4$ and for integers $r'\in[0, m-1]\backslash \{m-4\}$ and $0\leq r''\leq 2x-2$, a solution to $\mathrm{HWP}^{*}(v; 2^r, m^s)$ exists for each even $r \geq2$ except possibly $r=mx-4=v-4$ and $s \geq1$ satisfying $r+s=v-1$.

\textbf{(b)}  Assume $m\equiv 2 \pmod 4$. 

By Lemma \ref{wlackiodd}, factorize $K_{n}$ into $(\frac{n-1}{2})$ $C_{n}$-factors for odd $n$, and get a $C_{n}^{*}$-factorization of $K_{n}^{*}$ by Proposition \ref{lemma1.3}. Also, $K^{*}_{m}$ can be factorized as $K_{\frac{m}{2}}^{*}[2] \oplus I_{m}^*$. Since $\frac{m}{2}$ is odd, $K_{m}^{*}$ has a $\{(C_{\frac{m}{2}}^{*}[2])^{\frac{m-2}{4}},  I_{m}^*\}$-factorization. By Lemma \ref{lemma4}, each of $C_{\frac{m}{2}}^{*}[2]$-factors has a $\{(K_{2}^{*})^{r_0},\vv{C}_{m}^{s_0}\}$-factorization for $r_0\in \{0,2,4\}$ and $r_0+s_0=4$. By Lemma \ref{lemma2.6}, $C_{\frac{m}{2}}^{*}[2]\oplus I_{m}^*$ has $\{(K_{2}^{*})^{r_1},\vv{C}_{m}^{s_1}\}$-factorization for $r_1\in \{0,1,3,5\}$ and $r_1+s_1=5$.

Those factorizations give a $\{(K_{2}^{*})^{r_2}$, $ \vv{C}_{m}^{s_2}\}$-factorization of $K_{m}^{*}$ for $r_2=\frac{m-6}{4}r_0+r_1$ and $s_2=\frac{m-6}{4}s_0+s_1$ with $r_2+s_2=m-1$. 

Placing a $K_{2}^{*}$-factorization on $r'$ of the $K_{(\frac{m}{2}:2)}^{*}$ factors for $0\leq r'\leq 2x-2$ where we choose $r'$ is even, a $\vv{C}_{m}$-factorization on $s'$ of the $K_{(\frac{m}{2}:2)}^{*}$ with $r'+s'=2x-2$, and taking a $\{(K_{2}^{*})^{r_2}, \vv{C}_{m}^{s_2}\}$-factorization of $K_{m}^{*}$ give a $\{(K_{2}^{*})^{\frac{m}{2}r'+r_2},\vv{C}_{m}^{\frac{m}{2}s'+s_2}\}$-factorization of $K_{mx}^{*}$ where $r=\frac{m}{2}r'+r_2$ and $s=\frac{m}{2}s'+s_2$. Also, we obtain the requested even integer $ r\in [1, mx-1]$ except for $r=mx-4$, from the sum of $\frac{m}{2}r'$ and $r_2$ for integers $0\leq r'\leq 2x-2$ and $r_2\in [0, m-1]\backslash \{m-4\}$. So, a solution to $\mathrm{HWP}^{*}(v; 2^r, m^s)$ exists for even $r \geq 2$ except possibly $r=mx-4=v-4$ and  odd $s \geq 1$ satisfying $r+s=v-1$.

If $x$ is even, say $x=2t$, factorize $K^{*}_{mx}$ into a $K_{2m}^{*}$-factor and $(2t-2)$ $K_{(m:2)}^{*}$-factors. 
$K_{(m:2)}^{*}$ has a $K_{2}^{*}$-factorization with $m$ $K_{2}^{*}$-factors and a $\vv{C}_{m}$-factorization with $m$ $\vv{C}_{m}$-factors by Lemmata \ref{dliu2} and \ref{lemma2}, respectively. So, we must decompose $K_{2m}^{*}$ into $K_{2}^{*}$-factors and $\vv{C}_{m}$-factors. As before, $K_{2m}^{*}$ can be factorized as $K_{m}^{*}[2] \oplus I_{2m}^*$. So, $K_{2m}^{*}$ has a $\{(C_{m}^{*}[2])^{\frac{m-4}{2}}, I_{2m}^*, \Gamma^{*}_m\}$-factorization. By Lemma \ref{lemma2.4}, each of $C_{m}^{*}[2]$-factors has a $\{(K_{2}^{*})^{r_0},\vv{C}_{m}^{s_0}\}$-factorization for $r_0\in \{0,2,4\}$ and $r_0+s_0=4$. By Corollary \ref{lemma2.5}, $C_{m}^{*}[2]\oplus I_{2m}^*$ has a $\{(K_{2}^{*})^{r_1},\vv{C}_{m}^{s_1}\}$-factorization for $r_1\in \{1,3,5\}$ and $r_1+s_1=5$.
By Lemma \ref{lemma3.10}, $\Gamma^{*}_m$ has a $\{(K_{2}^{*})^{r_2}, \vv{C}_{m}^{s_2}\}$-factorization for $m\equiv 2 \pmod 4$ and $r_2\in\{1,2,3,4,6\}$ with $r_2+s_2=6$. Using these factorizations, we obtain a solution to the problem for $r=2mt-4=mx-4$ when $m\equiv 2 \pmod 4$ and even $x$. So, $\mathrm{HWP}^{*}(v; 2^r, m^s)$ has a solution for $r=v-4$ and even $\frac{v}{m}$ when $m\equiv 2 \pmod 4$.
\end{proof} 

\begin{lemma}\label{lemma5}
$C_{4}^{*}[2]\oplus I_8^*$ has a  $\{(K_{2}^{*})^{r}, \vv{C}_{4}^{s}\}$-factorization for $r\in \{0,1,2,3,5\}$ with $r+s=5$.
\end{lemma}  
\begin{proof}
Let $V(C_{4}^{*}[2]\oplus I_8^*)= \mathbb{Z}_{8}$ and we define the following $\vv{C}_{4}$-factorization of $C_{4}^{*}[2]\oplus I_8^*$;\\
\begin{footnotesize}
$
\mathcal{F}=\Big\{\big[(0,1,2,3),(4,5,6,7)\big],\big[(0,3,2,1),(4,7,6,5)\big],\big[(0,5,1,4),(2,7,3,6)\big],\big[(0,4,3,7),$\\
$(1,5,2,6)\big],\big[(0,7,2,5),(1,6,3,4)\big]\Big\}$
\end{footnotesize}

For $r=2$, we define the following $\vv{C}_{4}$-factors of $C_{4}^{*}[2]\oplus I_8^*$
\begin{footnotesize}
$$
F_{1}=[(0,1,2,3),(4,5,6,7)], \,\ F_{2}=[(0,3,6,5),(1,4,7,2)] , F_{3}=[(0,5,4,1),(2,7,6,3)]\,\
$$
\end{footnotesize}
and the following two $K_{2}^{*}$-factors,

\begin{footnotesize}
$
F_4=[(0,4)^*,(1,5)^*,(2,6)^*,(3,7)^*], \,\
F_5=[(0,7)^*,(1,6)^*,(2,5)^*,(3,4)^*]
$
\end{footnotesize}\\
Therefore $C_{4}^{*}[2]\oplus I_8^*$ has two $K_{2}^{*}$-factors and three $\vv{C_{4}}$-factors. 

The remaining cases are obtained from Corollary \ref{lemma2.5} for $m=4$.
\end{proof}

\begin{lemma}\label{lemma2.11}
$K_{12}^{*}$ has a $\{(K_{2}^{*})^{r}, \vv{C}_{4}^{s}\}$-factorization for $r\in \{0,1,2,3,4,5,7,9,\allowbreak 11\}$ with $r+s=11$.
\end{lemma} 
\begin{proof}
The cases $r=0$ and $r=11$ are obtained by Theorem \ref{ilk} and Proposition \ref{lemma3}, respectively. Since solutions to  $OP(4^{\frac{v}{4}})$ and  $OP(m^{\frac{v}{m}})$ exist except for $v=6$ or $v=12$ when $m=3$, $K_{12}-I$ has a $C_4$-factorization where $I$ is a $1$-factor of $K_{12}$. By Proposition \ref{lemma1.3}, $K_{12}^{*}$ can be factorized into five $C_4^{*}$-factors and one $I^*$-factor which is a $K_2^*$-factor of $K_{12}^{*}$. Also, $C_4^{*}$ has a $\vv{C}_{4}$-factorization and $K_{2}^{*}$-factorization. So, we obtain a $\{(K_{2}^{*})^{r}, \vv{C}_{4}^{s}\}$-factorization of $K_{12}^{*}$ for $r\in \{1,3,5,7,9\}$ with $r+s=11$.

Let $V(K_{12}^{*})$ be $\mathbb{Z}_{12}$, and define the following factorizations of $K_{12}^{*}$ for $r=2,4$, respectively. \\
\begin{footnotesize}
$
\mathcal{F}_{1}=\Big\{
[(0,6)^*,(1,7)^*,(2,8)^*,(3,9)^*,(4,10)^*,(5,11)^*],[(0,10)^*,(4,6)^*,(1,5)^*,(7,11)^*,
$\\
$
(2,9)^*,(3,8)^*],[(0,1,2,3),(4,5,6,7),(8,9,10,11)],
[(0,2,1,4),(3,5,7,6),(8,11,10,9)],
$\\
$
[(0,3,1,8),(2,4,11,6),(5,9,7,10)],
[(0,4,2,11),(1,6,8,10),(3,7,9,5)],
[(0,5,8,7),
$\\
$
(1,3,4,9),(2,10,6,11)],[(0,7,5,2),(1,10,8,4),(3,6,9,11)],[(0,8,6,1),(2,5,10,7),
$\\
$
(3,11,9,4)],[(0,9,6,5),(1,11,4,8),(2,7,3,10)],
[(0,11,1,9),(2,6,10,3),(4,7,8,5)]\Big\},
$
\end{footnotesize}
\\
\begin{footnotesize}
$
 \mathcal{F}_2=\Big\{
[(0,6)^*,(1,7)^*,(2,8)^*,(3,9)^*,(4,10)^*,(5,11)^*],[(0,10)^*,(4,6)^*,(1,5)^*,(7,11)^*,
$\\
$
(2,9)^*,(3,8)^*],[(0,8)^*,(2,6)^*,(1,10)^*,(4,7)^*,(3,11)^*,(5,9)^*],[(0,1)^*,(2,3)^*,(4,5)^*,
$\\
$
(6,7)^*,(8,9)^*,(10,11)^*],[(0,2,1,3),(4,8,11,9),(5,7,10,6)],[(0,3,10,5),(1,8,6,11),
$\\
$
(2,4,9,7)],[(0,4,11,2),(1,6,10,9),(3,5,8,7)],[(0,5,6,9),(1,2,11,4),(3,7,8,10)],
$\\
$
[(0,7,9,11),(1,4,3,6),(2,10,8,5)],[(0,9,10,7),(1,11,6,8),(2,5,3,4)],[(0,11,8,4),
$\\
$
(1,9,6,3),(2,7,5,10)]\Big\}.
$
\end{footnotesize}

So, $K_{12}^{*}$ has a $\{(K_{2}^{*})^{r}, \vv{C}_{4}^{s}\}$-factorization for $r\in \{0,1,2,3,4,5,7,$ $9,11\}$ with $r+s=11$.
\end{proof}

\begin{lemma}\label{lemma2.13}
$K_{(4:3)}^{*}$ has a $\{(K_{2}^{*})^{r}, \vv{C}_{4}^{s}\}$-factorization for $r\in \{0,1,2,4,6,8\}$ with $r+s=8$.
\end{lemma} 
\begin{proof}
The cases $r=0$ and $r=8$ are obtained by lemmata \ref{dliu} and \ref{dliu2}, respectively. By Theorem \ref{liu}, $K_{(4: 3)}$ has a $C_{4}$-factorization and so, $K_{(4: 3)}^{*}$ has a $C_{4}^*$-factorization by Proposition \ref{lemma1.3}. Since $C_{4}^*$ has a $K_{2}^*$-factorization and a $\vv{C}_{4}$-factorization, $K_{(4: 3)}$ can be factorized into two $K_{2}^*$-factors and six $\vv{C}_{4}$-factors. Similarly, a $\{(K_{2}^{*})^{r}, \vv{C}_{4}^{s}\}$-factorization of $K_{(4: 3)}$ is obtained for $r\in \{4,6\}$ with $r+s=8$.

Finally, let the vertex set of $K_{(4:3)}^{*}$ be $\mathbb{Z}_{12}$, and define the following factorization of $K_{(4:3)}^{*}$ for $r=1$,\\ 
\begin{footnotesize}
$
\mathcal{F}_1=\Big\{ [(0,4,2,5),(1,8,3,11),(6,9,7,10)],
[(0,5,1,7),(2,9,4,11),(3,8,6,10)],[(0,7,1,9),
$\\
$
(2,4,3,10),(5,11,6,8)],[(0,8,1,10),(2,7,3,5),(4,9,6,11)],[(0,9,2,11),(1,5,3,6),
$\\
$
(4,10,7,8)],[(0, 10, 4, 8),(1, 11, 5, 9),(2,6,3,7)],[(0,11,3,4),(1,6,2,10),(5,8,7,9)],[(0,6)^*,
$\\
$
(1,4)^*,(2,8)^*,(3,9)^*,(5,10)^*,(7,11)^*]\Big\}.
$
\end{footnotesize}
\end{proof} 
In Theorem \ref{maintheorem}, we have given the necessary and sufficient conditions for the existence of a solution for $\mathrm{HWP}^{*}(v; 2^{r}, m^{s})$ for even $m\geq 6$. The construction in Theorem \ref{maintheorem} is not valid when $m=4$, therefore we also examine the case of $m=4$ in the following theorem.
\begin{theorem}\label{maintheorem2}
Let $r$, $s$ be nonnegative integers. Then, $\mathrm{HWP}^{*}(v; 2^{r}, 4^{s})$ has a solution if and only if $r+s=v-1$ except for $s=1$ or $(r,v)=(0,4)$, and except possibly when at least one of the following conditions holds;
\begin{enumerate}
\item $r \geq 2$ even and $v\equiv 4, 20 \pmod {24}$,
\item $s\in \{3, 5\}$ and $v\equiv 12 \pmod {24}$.
\end{enumerate}
\end{theorem}
\begin{proof}  
If you remove $(v-2)$ disjoint $K_2^*$-factors from $K_v^*$, then the remaining factor must be a $K_2^*$-factor in $K_v^*$.  Thus, there is no  $\{(K_2^*)^{v-2}, \vv{C}_4^1 \}$-factorization of $K_v^*$. So, we may assume $s \neq 1$.

Since $\mathrm{HWP}^{*}(v; n^{r}, m^{s})$ has a solution for $r=0$ except for $(v,m)=(4,4)$ by Theorem \ref{OP}, HWP$^{*}(4; 2^{r}, 4^{s})$ has no solution for $r=0$. So, we may assume that $r\geq 1$.

\textbf{Case 1 (}{\boldmath$v\equiv 0 \pmod 8$\textbf{)}}

Let $v=8k$ for a positive integer $k$. Note that,  $K_{8k}^{*}$ can be factorized as $K_{4k}^{*}[2]\oplus I_{8k}^*$. Also, $K_{4k}^{*}[2]$ can be factorized into $C_{4}^{*}[2]$-factors and a $K_{2}^{*}[2]$-factor. The graph $k C_{4}^{*}[2]\oplus  I_{8k}^*$ can be considered as $(C_{4}^{*}[2]\oplus I_{8}^*)$-factor in $K_{8k}^{*}$. So, $K_{8k}^{*}$ has a $\{(C_{4}^{*}[2])^{2k-1}, I_{8}^*, K_{2}^{*}[2]\}$-factorization. Also, $C_{4}^{*}[2]$ has a $\{(K_{2}^{*})^{r_0},\vv{C}_{4}^{s_0}\}$-factorization for $r_0\in \{0,2,4\}$ where $r_0+s_0=4$ by Lemma \ref{lemma2.4}. Since $K_{2}^{*}[2]=C_{4}^{*}$, $K_{2}^{*}[2]$ has a $\{(K_{2}^{*})^{r_1},\vv{C}_{4}^{s_1}\}$-factorization for $r_1\in \{0,2\}$ and $r_1+s_1=2$. $C_{4}^{*}[2]\oplus I_{8}^*$ has a $\{(K_{2}^{*})^{r_2}, \vv{C}_{4}^{s_2}\}$-factorization for $r_2\in \{0,1,2,3,5\}$ where $r_2+s_2=5$ by Lemma \ref{lemma5}. These factorizations give a $\left\{\left(K_{2}^{*}\right)^{r}, \vv{C}_{4}^{s}\right\}$-factorization of $K_{8 k}^{*}$ for $r \neq 8 k-2$ with $r+s=8 k-1$.

Then, HWP$^{*}(v; 2^{r}, 4^{s})$ has a solution for  $r+s=v-1$, $s \neq 1$ and $v\equiv 0 \pmod 8$.

\textbf{Case 2 (}{\boldmath$v\equiv 4 \pmod 8$\textbf{)}}

Let $v=8k+4$ for a nonnegative integer $k$.

\textbf{(a)} Assume r is odd. Partition the vertices of $K_{8k+4}^{*}$ into $4k+2$ sets of size $2$, represent each set of size $2$ vertices in $K_{8k+4}^{*}$ with a single vertex and represent all double arcs between sets of size $2$ as a single double arc, to get a $K_{4k+2}^{*}$. By Proposition \ref{lemma3}, $K_{4k+2}^{*}$ has a decomposition into $4k+1$ $K_{2}^{*}$-factors. Construct a $K_{4}^{*}$-factor from one of the $K_{2}^{*}$-factors and a $K_{(2:2)}^{*}$-factor from each of the remaining $4k$ $K_{2}^{*}$-factors. Then, factorize $K_{8k+4}^*$ into a $K_{4}^{*}$-factor and $(4k)$ $K_{(2:2)}^{*}$-factors. $K_{4}^{*}$ has a decomposition into one $K_{2}^{*}$ and two $\vv{C}_{4}$-factors or three $K_{2}^{*}$-factors, and $K_{(2:2)}^{*}$ has a $\{\left(K_{2}^{*}\right)^{r_0}, \vv{C}_{4}^{s_0}\}$-factorization for $r_0\in\{0, 2\}$ satisfying $r+s=2$. So, $K_{8k+4}^{*}$ has a $\{(K_{2}^{*})^{r},\vv{C}_{4}^{s}\}$-factorization for odd $r$. Therefore, HWP$^{*}(v; 2^{r}, 4^{s})$ has a solution for odd $r$ and $v\equiv 4 \pmod 8$.

\textbf{(b)} Assume r is even, and also let $k\equiv 1 \pmod 3$. Then, we have $v=24l+12$ for some nonnegative integer $l$. 

Representing each part of $4$ vertices in $K_{24l+12}^{*}$ with a single vertex and all double arcs between parts of size $4$ as a single double arc, we have a $K_{6l+3}^{*}$. Since a Kirkman triple system exists for orders $6l+3$, we have a $C_3$-factorization of $K_{6l+3}$. Then, a $C_{3}^{*}$-factorization of $K_{6l+3}^{*}$ is obtained by Proposition \ref{lemma1.3}. 
 
Construct a $K_{12}^{*}$-factor from one of the $C_{3}^{*}$-factors and $K_{(4:3)}^{*}$-factor from each of the remaining $3l$ $C_{3}^{*}$-factors. Then, get a $\{ K_{12}^{*}, (K_{(4:3)}^{*})^{3l}\}$-factorizati-on of $K_{24l+12}^{*}$. By Lemma \ref{lemma2.11}, $K_{12}^{*}$ has a $\{(K_{2}^{*})^{r_0}, \vv{C}_{m}^{s_0}\}$-factorization $r_0\in \{0,1,2,3,4,$ $5,7,9,11\}$ with $r_0+s_0=11$. Also, $K_{(4:3)}^{*}$ has a $\{\left(K_{2}^{*}\right)^{r_1}, \vv{C}_{4}^{s_1}\}$-factorization by Lemma \ref{lemma2.13} for $r_1\in \{0,1,2,4,6,8\}$ with $r_1+s_1=8$. Those factorizations give a $\{(K_{2}^{*})^{r}, \vv{C}_{m}^{s}\}$-factorization of $K_{24l+12}^{*}$ where $r=r_0+ar_1$ and $s=s_0+bs_1$ satisfying $r+s=24l+11=v-1$ with $1\leq r, s \leq v-1$ and $a+b=3l$.
We obtain the requested even $r\in [0, v-1]$ except for $r=v-6$ and $r=v-4$, from the sum of $r_0$ and $ar_1$. Then, HWP$^{*}(v; 2^{r}, 4^{s})$ has a solution for  $r+s=v-1$, $s\notin \{3, 5\}$ and $v\equiv 12 \pmod {24}$.
\end{proof} 
\section{Solutions to HWP$^{*}(v;m^r, (2m)^s)$}
In this section, we prove that for even $m$, a solution to $\mathrm{HWP}^{*}(v; m^{r}, (2m)^{s})$ exists for $r+s=v-1$ and except possibly when $s\in\{1,3\}$. 

Firstly, factorize $K^{*}_{2mx}$ into a $K_{2m}^{*}$-factor and $(2x-2)$ $K_{(m:2)}^{*}$-factors. $K_{(m:2)}^{*}$ has a $\{\vv{C}_{m}^{r}, \vv{C}_{2m}^{s}\}$-factorization for $r\in \{0,m\}$ and $r+s=m$. Using Lemma \ref{wlackieven} and Proposition \ref{lemma1.3}, a $\{(C_{m}^{*}[2])^{\frac{m-4}{2}}, I_{2m}^*, \Gamma^{*}_m\}$-factorization of $K_{2m}^{*}$ is also obtained. Therefore, in order to factorize $K^{*}_{2mx}$ into $\vv{C}_{m}$-factors and $\vv{C}_{2m}$-factors, $\Gamma^{*}_m$, $C_{m}^{*}[2]\oplus I_{2m}^*$ and $C_{m}^{*}[2]$ must be factorized into $\vv{C}_{m}$-factors and $\vv{C}_{2m}$-factors. The following lemmata examine the existence of a $\{\vv{C}_{m}^{r}, \vv{C}_{2m}^{s}\}$-factorization of these graphs for $r+s\in \{4,5,6\}$.
\begin{lemma}\label{lemma4.1}

Let $m\geq4$ be an even integer. Then $\Gamma^{*}_m$ has a $\{\vv{C}_{m}^{r}, \vv{C}_{2m}^{s}\}$-factorization for $r\in \{0,6\}$ and $r+s=6$.
\end{lemma}
\begin{proof}
\textbf{Case 1 \boldmath{$(r=0)$}} By Lemma \ref{lemma2.8} $(i)$ and Proposition \ref{lemma1.3}, $\Gamma^{*}_m$ has a $\vv{C}_{2m}$-factorization. 

\textbf{Case 2 \boldmath{$(r=6)$}}  By Lemma \ref{lemma2.8} $(ii)$ and Proposition \ref{lemma1.3}, $\Gamma^{*}_m$ has a $\vv{C}_{m}$-factorization for  $m \equiv 0 \pmod 4$. 

When $m \equiv 2 \pmod 4$, define the following $m$-cycles. Also, let $\vv{C}_{m}^{(0)}$ and $\vv{C}_{m}^{(1)}$ cycles be equivalent to the $\vv{C}_{m}^{(0)}$ and $\vv{C}_{m}^{(2)}$ cycles respectively, as stated in Lemma \ref{lemma3.10}. \\
$
\vv{C}_{m}^{(2)}=\left(u_{0}, u_{1}, \ldots, u_{m-1}\right) \text { where } u_{i}= \begin{cases}(1, m-1-i) & \text { if } 0 \leq i \leq \frac{m}{2},  \\ (0, m-1-i) & \text { if } \frac{m}{2}+1 \leq i \leq m-1. \end{cases}
$
$\vv{C}_{m}^{(3)}=(y_{0}, y_{1}, \ldots y_{m-1} )$ where  $y_{0}=(0,0)$, $y_{1}=(0,\frac{m}{2})$, $y_{2}=(1,\frac{m}{2}+1)$, $y_{3}=(1,\frac{m}{2}-1)$  and 
$$
y_{i=}\left\{\begin{array}{l}
\left(1, \frac{m}{2}+(-1)^{i+1}\lfloor \frac{i}{2}\rfloor\right) \text {if } i \equiv 0, 1 \pmod4 \\
\left(0, \frac{m}{2}+(-1)^{i}\lfloor \frac{i}{2}\rfloor\right)  \,\ \text { if } i=2, 3\pmod4
\end{array}\right. \text { for } 4 \leq i\leq m-1.   
$$
 Using the above $m$-cycles, we obtain the following $m$-cycle factors. $F_{0}=\vec{C}_{m}^{(0)} \cup(\vv{C}_{m}^{(0)}+(1,0))$, $F_{1}=\vv{C}_{m}^{(1)} \cup R(\vv{C}_{m}^{(1)} +(1,0))$, $F_{2}=R\left(F_{1}\right)$, $F_{3}=\vv{C}_{m}^{(2)} \oplus (\vv{C}_{m}^{(2)}+(1, 0))$, $F_{4}=\vv{C}_{m}^{(3)} \cup (\vv{C}_{m}^{(3)}+(1, 0))$ and $F_{5}=\big(\Gamma^{*}_m\big)-\bigoplus_{i=0}^{4} F_{i}$. Then, $\{F_{0}, F_{1}, F_{2}, F_{3},$ $ F_{4},  F_{5}\}$ is a $\vv{C}_{m}$-factorization of $\Gamma^{*}_m$.
So, $\Gamma^{*}_m$ has a $\vv{C}_{m}$-factorization for even $m\geq4$.
\end{proof}

\begin{lemma}\label{lemma4.2}
Let $m\geq4$ be an even integer. Then $C_{m}^{*}[2]\oplus I_{2m}^*$ has a $\{\vv{C}_{m}^{r}, \vv{C}_{2m}^{s}\}$-factorization for $r\in \{1,3\}$ and $r+s=5$.
\end{lemma}
\begin{proof}
\textbf{Case 1 \boldmath{$(r=1)$}} Let $\vv{C}_{m}^{(0)}=(v_0,v_1, \dots,\allowbreak v_{m-1})$ be a directed $m$-cycle of $C_{m}^{*}[2]\oplus I_{2m}^*$, where $v_i=(0,i)$ for $0\leq i\leq m-1$, and it can be checked that $F_1=\vv{C}_{m}^{(0)}\cup (\vv{C}_{m}^{(0)}+(1,0))$ is a directed $m$-cycle factor of $C_{m}^{*}[2]\oplus I_{2m}^*$. Also, let $\vv{C}_{2m}^{(1)}=(u_0,u_1, \dots,u_{2m-1})$ be a directed $2m$-cycle of $C_{m}^{*}[2]\oplus I_{2m}^*$, where $u_{2i} = (0,i)$, and $u_{2i+1} = (1,i)$ for $0\leq i\leq m-1$. Similarly, it can be checked that $F_2=\vv{C}_{2m}^{(1)}$ and $F_3=\vv{C}_{2m}^{(1)}+(1,0)$ are arc disjoint directed $2m$-cycle factors of $C_{m}^{*}[2]\oplus I_{2m}^*$.

Let $\vv{C}_{2m}^{(2)}=(x_0,x_1, \dots,\allowbreak x_{2m-1})$ be a directed $2m$-cycle of $C_{m}^{*}[2]\oplus I_{2m}^*$, where $x_0=(0,0)$, $x_m=(1,0)$, $x_{i+1}=(0,m-1-i)$ for $0\leq i\leq m-2$ and  $x_{j+1+m}=(1, m-1-j)$ for $0\leq j\leq m-2$. 

$F_4=\vec{C}_{2m}^{(2)}$ and $F_5=(C_{m}^{*}[2]\oplus I_{2m}^*)-\bigoplus_{i=0}^{4} F_{i}$ are arc disjoint directed $2m$-cycle factors of $C_{m}^{*}[2]\oplus I_{2m}^*$. Then, $\left\{F_{1}, F_{2}, F_{3}, F_{4},  F_{5}\right\}$ is a $\{\vv{C}_{m}^{1}, \vv{C}_{2m}^{4}\}$-factorization of $C_{m}^{*}[2]\oplus I_{2m}^*$.

\textbf{Case 2 \boldmath{$(r=3)$}} Let $F_{1}$, $F_{2}$ and $F_{3}$ be the same as in Case 1. Using the arcs of $F_{4} \bigcup F_{5}$, we obtain two new $\vv{C}_{m}$-factors.

$F_{4}^{'} = R(F_{1})$ is a $\vv{C}_{m}$-factor of $C_{m}^{*}[2]\oplus I_{2m}^*$. Let $\vv{C}=(y_0,y_1, \dots,\allowbreak y_{m-1})$ be a directed $m$-cycle of $C_{m}^{*}[2]\oplus I_{2m}^*$, where 
$$ y_{i} =
\begin{cases}
      (0,i) & if \,\   i \,\ is \,\  even\\
      (1,i) & if \,\  i \,\  is \,\  odd
\end{cases}  \,\ \text{for} \,\ 0\leq i\leq m-1. 
$$
It can be checked that $F_5^{'}=R(\vv{C})\cup R(\vv{C}+(1,0))$ is a directed $m$-cycle factor of $C_{m}^{*}[2]\oplus I_{2m}^*$. 

So, $\left\{F_{1}, F_{2}, F_{3}, F_{4}^{'},  F_{5}^{'}\right\}$ is a $\{\vv{C}_{m}^{3}, \vv{C}_{2m}^{2}\}$-factorization of $C_{m}^{*}[2]\oplus I_{2m}^*$.
\end{proof}

\begin{lemma}\label{lemma4.3}
Let $m\geq4$ be an even integer. Then $C_{m}^{*}[2]$ has a $\{\vv{C}_{m}^{r}, \vv{C}_{2m}^{s}\}$-factorization for $r\in \{0,2,4\}$ and $r+s=4$.
\end{lemma}
\begin{proof}
The cases $r\in \{0,4\}$ are obtained by lemmata \ref{lemma4} and \ref{lemma2.4}. Let $\vv{C}_{2m}^{(1)}=(u_0,u_1, \dots,\allowbreak u_{2m-1})$ be a directed $2m$-cycle of $C_{m}^{*}[2]$, where 
$$ u_i =
\begin{cases}
      (0,i) & if \,\  0\leq i\leq m-1,\\
      (1,i) & if \,\  m\leq i\leq 2m-1.
\end{cases}$$
And it can be checked that $F_1=\vv{C}_{2m}^{(1)}$ is a $\vv{C}_{2m}$-factor of $C_{m}^{*}[2]$.  Let $\vv{C}_{2m}^{(2)}=(v_0,v_1, \dots,\allowbreak v_{2m-1})$ be a directed $2m$-cycle of $C_{m}^{*}[2]$, where 
$$ v_i =
\begin{cases}
      u_i & if \,\  i \,\  is \,\  even, \\
     u_i+(1,0) & if \,\  i \,\  is \,\  odd.
\end{cases}$$
$F_{2}=\vv{C}_{2m}^{(2)}$ is a $\vv{C}_{2m}$-factor of $C_{m}^{*}[2]$. Then, $\left\{F_{1}, F_{2}, F_{4}^{'},  F_{5}^{'}\right\}$ is a $\{\vv{C}_{m}^{2}, \vv{C}_{2m}^{2}\}$-factorization of $C_{m}^{*}[2]$ where $F_{4}^{'}$ and $F_{5}^{'}$ are the same factors in Lemma \ref{lemma4.2}.
\end{proof}
\begin{theorem}
Let $r$, $s$ be nonnegative integers, and let $m\geq 4$ be even. Then, $\mathrm{HWP}^{*}(v; m^{r}, (2m)^{s})$ has a solution if and only if $m| v$, $r+s=v-1$ and $v\geq 4$ except for $(s,v,m)\in\{(0,4,4),(0,6,3),(5,6,6)\}$, and except possibly when $s\in\{1,3\}$.
\end{theorem}
\begin{proof}
By Theorem \ref{ilk}, $\mathrm{HWP}^{*}(v; 4^{r}, 8^{s})$ has a solution for $r+s=v-1$, so we may assume that $m\geq 6$. Furthermore, by Theorem \ref{OP}, a solution to the $\mathrm{HWP}^{*}(v; m^{r}, (2m)^{s})$ exists for $r=s=0$ and  except for $(s,v,m)\in\{(0,4,4),$ $(0,6,3),(5,6,6)\}$. 

Factorize $K^{*}_{2mx}$ into a $K_{2m}^{*}$-factor and $(2x-2)$ $K_{(m:2)}^{*}$-factors. $K_{(m:2)}^{*}$ decomposes into $m$ $\vv{C}_{m}$-factors or $m$ $\vv{C}_{2m}$-factors by Lemma \ref{dliu}. So, $K_{2m}^{*}$ must be decomposed into $\vv{C}_{m}$-factors and $\vv{C}_{2m}$-factors. As before, $K_{2m}^{*}$ can be factorized as $K_{m}^{*}[2] \oplus I_{2m}^*$. So, $K_{2m}^{*}$ has a $\{(C_{m}^{*}[2])^{\frac{m-4}{2}}, I_{2m}^*, \Gamma^{*}_m\}$-factorization. By Lemma \ref{lemma4.3}, each of $C_{m}^{*}[2]$-factors has a $\{\vv{C}_{m}^{r_0},\vv{C}_{2m}^{s_0}\}$-factorization for $r_0\in \{0,2,4\}$ and $r_0+s_0=4$. By lemmata \ref{lemma4.2} and \ref{lemma2.6}, $C_{m}^{*}[2]\oplus I_{2m}^*$ has a $\{\vv{C}_{m}^{r_1},\vv{C}_{2m}^{s_1}\}$-factorization for $r_1\in \{0,1,3\}$ and $r_1+s_1=5$. By Lemma \ref{lemma4.1}, $\Gamma^{*}_m$ has a $\{\vv{C}_{m}^{r_2}, \vv{C}_{2m}^{s_2}\}$-factorization for  $r_2\in \{0,6\}$ with $r_2+s_2=6$. Those factorizations give a $\{\vv{C}_{m}^{r}, \vv{C}_{2m}^{s}\}$-factorization of $K_{2m}^{*}$ where $r=(\frac{m-6}{2})r_0+r_1+r_2$ and $s=(\frac{m-6}{2})s_0+s_1+s_2$ satisfying $r+s=(\frac{m-6}{2})4+5+6=2m-1$ with $0\leq r, s \leq 2m-1$ and $s\notin \{1,3\}$.

Placing a $\vv{C}_{m}$-factorization on $r'$ of the $K_{(m:2)}^{*}$-factors for $0\leq r'\leq 2x-2$, a $\vv{C}_{2m}$-factorization on $s'$ of the $K_{(m:2)}^{*}$ for $r'+s'=2x-2$, and taking a $\{\vv{C}_{m}^{r},\vv{C}_{2m}^{s}\}$-factorization of $K_{2m}^{*}$ give a $\{\vv{C}_{m}^{mr'+r},\vv{C}_{2m}^{ms'+s}\}$-factorization of $K_{2mx}^{*}$. Then, $\mathrm{HWP}^{*}(v; m^r, 2m^s)$ has a solution except possibly when $s\in \{1,3\}$.
\end{proof}

\section{Conclusions}
Combining the results of  Theorem \ref{maintheorem} and \ref{maintheorem2}, we obtain one of the main results of this paper.
\begin{theorem}
Let $r$, $s$ be nonnegative integers, and let $m\geq 4$ be even. Then, $\mathrm{HWP}^{*}(v; 2^{r}, m^{s})$ has a solution if $m| v$, $r+s=v-1$, $s \neq 1$, $(r, v)\neq(0, 6)$, $(m, r, v)\neq (4,0,4)$, and one of the following conditions holds;
\begin{enumerate}
\item $m > 4$, $s\neq 3$ and $m\equiv 0 \pmod 4$,
\item $m > 4$, $\frac{v}{m}$ is even, $s\neq 3$ and $m\equiv 2 \pmod 4$,
\item $m=4$ and $v\equiv 0,8,16 \pmod {24}$,
\item $m=4$, $v\equiv 12 \pmod {24}$ and $s\notin \{3, 5\}$,
\item $m=4$, $v\equiv 4, 20 \pmod {24}$ and $r$ is odd.
\end{enumerate}
\end{theorem}

In this paper we also show that $K^{*}_{2mx}$  has a $\{\vv{C}_{m}^{r}, \vv{C}_{2m}^{s}\}$-factorization for $r+s=2mx-1$, which means that the solution of $\mathrm{HWP}^{*}(2mx; m^{r}, (2m)^{s})$ exists.
\begin{theorem}
Let $r$, $s$ be nonnegative integers, and let $m\geq 4$ be even. Then, $\mathrm{HWP}^{*}(v; m^{r}, (2m)^{s})$ has a solution if and only if $m|v$, $r+s=v-1$ and $v\geq 4$ except for $(s,v,m)\in\{(0,4,4),(0,6,3),(5,6,6)\}$, and except possibly when $s\in\{1,3\}$.
\end{theorem}

Actually, since $\mathrm{HWP}^{*}(v; 2^{r}, m^{s})$ has a solution with a
few possible exceptions, this result can be extended to $m\geq 2$.

\end{document}